\documentclass{amsart}

\usepackage{amsmath,amssymb,amsthm}
\usepackage{mathtools}
\usepackage{url}
\usepackage[colorlinks]%
{hyperref}

\usepackage[colorinlistoftodos, bordercolor=orange,
backgroundcolor=orange!20, linecolor=orange,
textsize=scriptsize]{todonotes}
\definecolor{cJakob}{rgb}{0.1,0.45,0.03}

\definecolor{cChristian}{rgb}{0.45,0.03,0.1}

\definecolor{cPaco}{rgb}{0.03,0.1,0.45}

\newtheorem{theorem}{Theorem} %[section]
\newtheorem{corollary}[theorem]{Corollary}
\newtheorem{proposition}[theorem]{Proposition}
\newtheorem{lemma}[theorem]{Lemma}

\theoremstyle{definition}

\newtheorem{definition}[theorem]{Definition}
\newtheorem{remark}[theorem]{Remark}

\newtheorem*{mainquestion}{Main Question}
\newtheorem{example}[theorem]{Example}

\newcommand{\Z}{\mathbb{Z}}
\newcommand{\N}{\Z_{\ge 1}}
\newcommand{\R}{\mathbb{R}}

\newcommand{\Q}{\mathbb{Q}}
\newcommand{\OO}{\mathcal{O}}
\DeclareMathOperator{\glzOperator}{GL}
\newcommand{\glz}[1][3]{\glzOperator_{#1}(\Z)}

\DeclareMathOperator{\affOperator}{Aff}
\newcommand{\affinegroup}[1][3]{\affOperator_{#1}(\Z)}
\DeclareMathOperator{\conv}{conv}
\DeclareMathOperator{\Ehr}{ehr}
\DeclareMathOperator{\affinehull}{aff}
\DeclareMathOperator{\relint}{relint}
\newcommand{\lattice}{\Lambda}
\newcommand{\dissect}{\bigsqcup}
\newcommand{\suchthat}{\,\colon\,}

\newcommand{\tetra}{T}

\title[Ehrhart-equivalent lattice $3$-polytopes are equidecomposable]
{Ehrhart-equivalent $\boldsymbol 3$-polytopes are equidecomposable}

\author{Jakob Erbe}
\address{Jakob Erbe \\ 
  Freie Universit\"at Berlin \\ 
  14195 Berlin \\ Germany}
\email{jakob.erbe@web.de}

\author{Christian Haase}
\address{Christian Haase \\ 
  Freie Universit\"at Berlin \\ 
  14195 Berlin \\ Germany}
\email{haase@math.fu-berlin.de}
\thanks{C.~H.~is supported by the research training group Facets of
  Complexity GRK 2434 of the German Research Foundation DFG} 

\author{Francisco Santos}
\address{Francisco Santos \\ 
Departamento de Matem\'aticas, Estad\'istica y Computaci\'on,
Universidad de Cantabria,
39005 Santander, Spain}
\email{francisco.santos@unican.es}
\thanks{F.~S.~is supported by grant MTM2014-54207-P of the Spanish
  Ministry of Science and grant EVF-2015-230 of the Einstein Foundation Berlin}

\begin{document}

\maketitle

\begin{abstract}
We show that if two lattice $3$-polytopes $P$ and $P'$ have the same Ehrhart function then they are \emph{\( \glz[3] \)-equidecomposable}; that is, they can be partitioned into relatively open simplices $U_1,\dots, U_k$ and $U'_1,\dots,U'_k$ such that $U_i$ and $U'_i$ are unimodularly equivalent, for each $i$.
\end{abstract}

\section{Introduction}
\label{sec:intro}

\subsection{Motivation}
Consider the rational polygons $P$ with vertices $\big\{\binom{-4}{0}, \binom{-1}{0}, \binom{-3}{\frac23}\big\} $ and $P'$ with vertices $\left\{\binom{1}{0}, \binom{3}{0}, \binom{1}{1}\right\} $, depicted here:
\begin{center}
\begin{tikzpicture}[y=-1cm]

% objects at depth 70:
\definecolor{penColor}{rgb}{0.52941,0.80784,1}
\path[draw=penColor,semithick,join=round,fill=white!75!black,arrows=-to] (1.5,6) -- (9.7,6);
\path[draw=penColor,semithick,join=round,fill=white!75!black,arrows=-to] (6,6.5) -- (6,4.3);

% objects at depth 60:
\path[draw=black,thick,join=round,fill=white!75!black] (2,6) -- (3,5.5) -- (5,6) -- cycle;
\path[draw=black,thick,join=round,fill=white!75!black] (7,6) -- (7,5) -- (9,6) -- cycle;
\path (3.3,6.6) node[text=black,anchor=base west] {\large{}$P$};
\path (7.6,6.6) node[text=black,anchor=base west] {\large{}$P'$};

% objects at depth 50:
\path[draw=penColor,semithick,join=round,fill=white!75!black] (5,5.9) -- (5,6.1);
\path[draw=penColor,semithick,join=round,fill=white!75!black] (5.1,6) -- (4.9,6);
\path[draw=penColor,semithick,join=round,fill=white!75!black] (4,5.9) -- (4,6.1);
\path[draw=penColor,semithick,join=round,fill=white!75!black] (4.1,6) -- (3.9,6);
\path[draw=penColor,semithick,join=round,fill=white!75!black] (3,5.9) -- (3,6.1);
\path[draw=penColor,semithick,join=round,fill=white!75!black] (3.1,6) -- (2.9,6);
\path[draw=penColor,semithick,join=round,fill=white!75!black] (2,5.9) -- (2,6.1);
\path[draw=penColor,semithick,join=round,fill=white!75!black] (2.1,6) -- (1.9,6);
\path[draw=penColor,semithick,join=round,fill=white!75!black] (2,4.9) -- (2,5.1);
\path[draw=penColor,semithick,join=round,fill=white!75!black] (2.1,5) -- (1.9,5);
\path[draw=penColor,semithick,join=round,fill=white!75!black] (3,4.9) -- (3,5.1);
\path[draw=penColor,semithick,join=round,fill=white!75!black] (3.1,5) -- (2.9,5);
\path[draw=penColor,semithick,join=round,fill=white!75!black] (4,4.9) -- (4,5.1);
\path[draw=penColor,semithick,join=round,fill=white!75!black] (4.1,5) -- (3.9,5);
\path[draw=penColor,semithick,join=round,fill=white!75!black] (5,4.9) -- (5,5.1);
\path[draw=penColor,semithick,join=round,fill=white!75!black] (5.1,5) -- (4.9,5);
\path[draw=penColor,semithick,join=round,fill=white!75!black] (6,4.9) -- (6,5.1);
\path[draw=penColor,semithick,join=round,fill=white!75!black] (6.1,5) -- (5.9,5);
\path[draw=penColor,semithick,join=round,fill=white!75!black] (7,4.9) -- (7,5.1);
\path[draw=penColor,semithick,join=round,fill=white!75!black] (7.1,5) -- (6.9,5);
\path[draw=penColor,semithick,join=round,fill=white!75!black] (8,4.9) -- (8,5.1);
\path[draw=penColor,semithick,join=round,fill=white!75!black] (8.1,5) -- (7.9,5);
\path[draw=penColor,semithick,join=round,fill=white!75!black] (9,4.9) -- (9,5.1);
\path[draw=penColor,semithick,join=round,fill=white!75!black] (9.1,5) -- (8.9,5);
\path[draw=penColor,semithick,join=round,fill=white!75!black] (9,5.9) -- (9,6.1);
\path[draw=penColor,semithick,join=round,fill=white!75!black] (9.1,6) -- (8.9,6);
\path[draw=penColor,semithick,join=round,fill=white!75!black] (8,5.9) -- (8,6.1);
\path[draw=penColor,semithick,join=round,fill=white!75!black] (8.1,6) -- (7.9,6);
\path[draw=penColor,semithick,join=round,fill=white!75!black] (7,5.9) -- (7,6.1);
\path[draw=penColor,semithick,join=round,fill=white!75!black] (7.1,6) -- (6.9,6);

\end{tikzpicture}%

\end{center}
We claim that for every $k \in \Z_{>0}$, the dilations $kP$ and $kP'$
contain the same number of integer points: $|kP \cap \Z^2| = |kP' \cap
\Z^2|$. One way to see this is to compute this number for every
$k$, which can be done for example using Ehrhart theory.
A more insightful argument is to decompose $P$ and $P'$ as follows:
\begin{center}
\begin{tikzpicture}[y=-1cm]

% objects at depth 70:
\definecolor{penColor}{rgb}{0.52941,0.80784,1}
\path[draw=penColor,semithick,join=round,fill=white!75!black,arrows=-to] (1.5,6) -- (9.7,6);
\path[draw=penColor,semithick,join=round,fill=white!75!black,arrows=-to] (6,6.5) -- (6,4.3);
\path (2.3,6.5) node[anchor=base west] {$Q_1$};
\path (5.5,7.2) node[anchor=base west] {\footnotesize{}$x \mapsto x + \binom{4}{0}$};
\path (3.3,4) node[anchor=base west]
  {\footnotesize{}$x \mapsto
    \left(\begin{smallmatrix}
      2&-3\\-1&1
      \end{smallmatrix}\right)
    x + \binom{9}{-3}
  $};
\path (8.1,5.4) node[anchor=base west] {$Q_1'$};
\path (3.7,6.5) node[text=black,anchor=base west] {$Q_2$};
\path (7.5,6.5) node[text=black,anchor=base west] {$Q_2'$};

% objects at depth 60:
\path[draw=black,thick,join=round,fill=white!75!black] (3,5.3) -- (5,6) -- (3,6) -- cycle;
\path[draw=black,thick,join=round,fill=white!75!black] (7,5.4) -- (9,6) -- (7,6) -- cycle;
\path[draw=black,thick,join=round,fill=white!75!black] (2.9,5.4) -- (2,6) -- (2.9,6);
\path[draw=black,semithick,join=round,fill=white!75!black,dotted] (2.9,5.4) -- (2.9,6);
\path[draw=black,thick,join=round,fill=white!75!black] (7,5.3) -- (7,5) -- (9,5.9);
\path[draw=black,semithick,join=round,fill=white!75!black,dotted] (7,5.3) -- (9,5.9);
% \draw[semithick,arrows=-to,black] (2.5,5) -- (2.5,4.3) -- (4.7,3.9) -- (6.9,4.6);
% \draw[semithick,arrows=-to,black] (5.1,6.3) .. controls (6.4,7.4) .. (7.4,6.3);

% objects at depth 50:
\path[draw=penColor,semithick,join=round,fill=white!75!black] (5,5.9) -- (5,6.1);
\path[draw=penColor,semithick,join=round,fill=white!75!black] (5.1,6) -- (4.9,6);
\path[draw=penColor,semithick,join=round,fill=white!75!black] (4,5.9) -- (4,6.1);
\path[draw=penColor,semithick,join=round,fill=white!75!black] (4.1,6) -- (3.9,6);
\path[draw=penColor,semithick,join=round,fill=white!75!black] (3,5.9) -- (3,6.1);
\path[draw=penColor,semithick,join=round,fill=white!75!black] (3.1,6) -- (2.9,6);
\path[draw=penColor,semithick,join=round,fill=white!75!black] (2,5.9) -- (2,6.1);
\path[draw=penColor,semithick,join=round,fill=white!75!black] (2.1,6) -- (1.9,6);
\path[draw=penColor,semithick,join=round,fill=white!75!black] (2,4.9) -- (2,5.1);
\path[draw=penColor,semithick,join=round,fill=white!75!black] (2.1,5) -- (1.9,5);
\path[draw=penColor,semithick,join=round,fill=white!75!black] (3,4.9) -- (3,5.1);
\path[draw=penColor,semithick,join=round,fill=white!75!black] (3.1,5) -- (2.9,5);
\path[draw=penColor,semithick,join=round,fill=white!75!black] (4,4.9) -- (4,5.1);
\path[draw=penColor,semithick,join=round,fill=white!75!black] (4.1,5) -- (3.9,5);
\path[draw=penColor,semithick,join=round,fill=white!75!black] (5,4.9) -- (5,5.1);
\path[draw=penColor,semithick,join=round,fill=white!75!black] (5.1,5) -- (4.9,5);
\path[draw=penColor,semithick,join=round,fill=white!75!black] (6,4.9) -- (6,5.1);
\path[draw=penColor,semithick,join=round,fill=white!75!black] (6.1,5) -- (5.9,5);
\path[draw=penColor,semithick,join=round,fill=white!75!black] (8,4.9) -- (8,5.1);
\path[draw=penColor,semithick,join=round,fill=white!75!black] (8.1,5) -- (7.9,5);
\path[draw=penColor,semithick,join=round,fill=white!75!black] (9,4.9) -- (9,5.1);
\path[draw=penColor,semithick,join=round,fill=white!75!black] (9.1,5) -- (8.9,5);
\path[draw=penColor,semithick,join=round,fill=white!75!black] (8,5.9) -- (8,6.1);
\path[draw=penColor,semithick,join=round,fill=white!75!black] (8.1,6) -- (7.9,6);
\path[draw=penColor,semithick,join=round,fill=white!75!black] (7,5.9) -- (7,6.1);
\path[draw=penColor,semithick,join=round,fill=white!75!black] (7.1,6) -- (6.9,6);
\path[draw=penColor,semithick,join=round,fill=white!75!black] (9,5.9) -- (9,6.1);
\path[draw=penColor,semithick,join=round,fill=white!75!black] (9.1,6) -- (8.9,6);
\path[draw=penColor,semithick,join=round,fill=white!75!black] (7,4.9) -- (7,5.1);
\path[draw=penColor,semithick,join=round,fill=white!75!black] (7.1,5) -- (6.9,5);

\path (2.5,5.5) node (Q1) {};
\path (6.8,5) node (Q1p) {};
\draw (Q1) [out=95, in=150, ->] to (Q1p);

(5.1,6.3) .. controls (6.4,7.4) .. (7.4,6.3)
\path (4.9,6.3) node (Q2) {};
\path (7.4,6.3) node (Q2p) {};
\draw (Q2) [out=330, in=210, ->] to (Q2p);
\end{tikzpicture}%

\end{center}
These decompositions yield a bijection between the rational
points in $P$ and $P'$ that preserves denominators, hence a bijection $kP \cap \Z^2 \leftrightarrow kP' \cap \Z^2$ for every $k$.
In this paper we address the question of whether every \emph{Ehrhart-equivalence} admits an explanation via \emph{equipartitions}, as happens in the example.  We show the answer to be positive for $3$-dimensional lattice polytopes.

This question is reminiscent of Hilbert's third problem about the equidecomposability of polytopes of the same volume into pieces that can be congruently bijected.
Instead of the volume, the valuation under consideration is the Ehrhart polynomial; instead of rigid motions we consider unimodular transformations.

\subsection{Framework}
We study decompositions of a lattice polytope $P$ into rational polytopes.
That is, $P \subset \R^d$ is the convex hull of finitely many points
in \( \Z^{d} \), and each piece $Q$ in the decomposition will be the
convex hull of finitely many points in \( \Q^{d} \).

The \emph{denominator} of a retional polytope \( Q \) is the minimum positive
integer \( D\in \Z_{\ge 1} \) such that the dilation \( D Q \) is a lattice
polytope.
The dimension \( \dim Q \) of $Q$ is the dimension of the affine
subspace \( \affinehull Q \subseteq \R^{d} \) spanned by \( Q \).
The relative interior $\relint Q$ of $Q$ is the interior of $Q$ inside
$\affinehull Q$. Similarly, any subset of $\R^d$ is relatively open if
it is open inside its affine span.

Ehrhart's Theorem~\cite[Th\'eor\`eme~38]{EhrhartBook} states that the
counting function\\
\( k \mapsto \lvert kQ \cap \Z^d \rvert \) agrees for $k \in \Z_{\ge 1}$
with a quasipolynomial $\Ehr(Q;k)$ of degree $\dim Q$ with period $D$. That is, 
for every $k$ we have
\( \Ehr(Q;k) = c_{\dim Q}(k) k^{\dim Q} + \dotsc + c_{1}(k) k + c_{0}(k) \)
with periodic functions \( c_{0}, \dotsc, c_{\dim Q} \colon \Z \to \Q \)
such that \( c_{\dim Q} \) is not identically zero and \(
c_i(k+D)=c_i(k) \)
for all $i$ and $k$. In particular, if $P$ is a lattice polytope then \( \Ehr(P;k) \) is an honest
polynomial. 

\begin{definition}
\label{defi:equivalent}
Two rational polytopes $Q_1, Q_2 \in \R^d$ are called \emph{Ehrhart-equivalent}
if they have the same Ehrhart quasipolynomial. That is, if \( \lvert
kQ_1 \cap \Z^d \rvert = \lvert kQ_2 \cap \Z^d \rvert \) for all $k \in
\N$.
\end{definition}

\begin{example} \label{example:intervals}
  Consider the $1$-dimensional polytopes \( Q = [0,1] \), \( Q' =
  [1/5,6/5] \) and \( Q'' = [2/5,7/5] \). Then
  \( \Ehr(Q;k) = k+1 \) while
  \( \Ehr(Q';k) =
  \Ehr(Q'';k) = k + c_0(k) \) where 
  \( c_0(k) = 1 \) if \( k \in 5\Z \) and \( c_0(k) = 0 \) else. So
  \(Q' \) and \(Q'' \) are Ehrhart-equivalent to each other but not to
  \(Q\).
\end{example}

The Ehrhart quasipolynomial is invariant under the group $\affinegroup[d] := \glz[d]
\ltimes \Z^{d}$ of lattice preserving affine maps. We call such maps
\emph{unimodular transformations}.
Now we can define what kind of ``nice explanation'' for
Ehrhart-equivalence we seek.

\begin{definition}
  \label{def:slz-equivalent}
  We say that two polytopes $P, Q \subset \R^d$ are
  \emph{$\glz[]$-equi\-de\-com\-posable}
  if there are relatively open simplices 
  \( T_{1}, \dotsc, T_{r} \) and
  unimodular transformations \( U_{1}, \dotsc, U_{r} \in
  \glz[d] \ltimes \Z^{d} \) such that
  \begin{equation*}
    P = \dissect_{i=1}^{r} T_{i}
    \quad \text{and} \quad
    Q = \dissect_{i=1}^{r} U_{i}(T_{i}).
  \end{equation*}
  (Here, \( \dissect \) indicates disjoint union.)
\end{definition}

It would be nice if the converse was true.

\begin{mainquestion}[\protect{\cite[Question~4.1]{HaaseMcAllister}}]
\label{question:main}
Is it true that every pair of Ehrhart-equivalent polytopes are \( \glz[] \)-equidecomposable?
\end{mainquestion}

One case where this is true is when both $P$ and $P'$ admit a unimodular triangulation, that is, a triangulation into simplices that are \emph{$\glz[]$-equivalent} to the standard simplex $\conv(0,e_1,\dots,e_k)$. In this case the Ehrhart quasipolynomial contains the same information as the $f$-vector of such triangulations~\cite{BetkeMcmullen}, and unimodular triangulations with the same $f$-vector clearly yield a \( \glz[] \)-equidecomposition.

This in particular implies a positive answer to the Main Question in dimension two, since all lattice polygons have unimodular triangulations.
Peter Greenberg proved an even stronger statement~\cite[Theorem~2.4]{Greenberg}: 
Ehrhart-equivalent lattice polygons can be related to one another by a sequence of \( \glz[] \)-equidecompositions of a particular type that he calls $1$-triangulated homeomorphisms.
Imre B\'ar\'any and Jean-Michel Kantor ask a similar question under
the stronger hypothesis that $|P \cap \lattice| = |P' \cap \lattice|$
for every super lattice $\lattice \supseteq \Z^d$~\cite{BaranyKantor}.

In dimension $3$ existence of unimodular triangulations does not hold for every lattice polytope. Even more, the following two polytopes $P$ and $P'$ have the same Ehrhart polynomial but $P'$ admits a unimodular triangulation while $P$ does not:
\begin{equation*}
  P = \conv \left[
    \begin{smallmatrix}
      0&1&0&1&2&1\\
      0&0&1&1&1&2\\
      0&0&0&3&3&3
    \end{smallmatrix}
  \right]
  \quad \text{and} \quad
  P' = \conv \left[
    \begin{smallmatrix}
      1&-1&0& 0&1&1\\
      0& 0&1&-1&1&0\\
      0& 0&0& 0&1&-1
    \end{smallmatrix}
  \right].
\end{equation*}
Jean-Michel Kantor conjectures that, in general dimension,  one cannot even find a piece-wise unimodular homeomorphism~\cite[p.~212]{Kantor} between every pair of Ehrhart equivalent lattice polytopes.

It is worth pointing out that the answer to the Main Question turns
out to be negative if we extend it to rational polytopes, even in
dimension $1$.
\begin{example} \label{example:intervals:continued}
  The polytopes \( Q' = [1/5,6/5] \) and \( Q'' = [2/5,7/5] \) from
  Example~\ref{example:intervals} are Ehrhart equivalent but they
  cannot be  \( \glz[] \)-equidecomposable: $Q'$ contains three 
  points from the \( \affinegroup[1] \)-orbit of $1/5$ (namely, $1/5$,
  $4/5$ and $6/5$) while $Q''$ contains only two ($4/5$ and
  $6/5$). See \cite{TurnerWuAlgorithm,TurnerWuWeightInvariant} for
  more examples in dimension $2$.
\end{example}

\begin{remark}
 What happens in the above example is that all points of the form $a/5$ with $a\in \Z\setminus 5\Z$ have the same Ehrhart function, but are not in the same $\affinegroup[1]$-orbit. This can be formalized as follows:

 Let \( \lattice = \frac1q \Z^d \) be a super lattice of $\Z^d$. The group \(
 \affinegroup[d] \) of unimodular transformations acts on the cosets \(
 \lattice/\Z^d \). Denote by \( \OO = \OO(\lattice) = 
 {\scriptstyle\affinegroup[d]} \backslash \lattice / {\scriptstyle\Z^d} \) the
 set of orbits of this action, so that \( [\lambda] \in \OO \) is the orbit
 of \( \lambda \in \lattice \). We  identify an orbit \( o \in \OO
 \) with the corresponding set of \( \lattice \)-points
 \( \{ \lambda \in \lattice \suchthat [\lambda] = o \} \). Then
 \( \glz[] \)-equidecomposable polytopes \( P \), \( P' \) satisfy
 \( | P \cap o | = | P' \cap o | \) for all \( \lattice \) and
 all \(o\). (This approach is different from \cite[\S1.3]{Kantor}.)

   For \( \lattice = \frac15 \Z \supseteq \Z \), we get three orbits \(
   \OO = \{ \Z \,,\, \{\frac15,\frac45\}+\Z \,,\, \{\frac25,\frac35+\Z \} \). The \( 0 \)-dimensional polytopes  \(\{1/5\} \) and \(\{2/5\}\) are Ehrhart-equivalent but not \( \glz[]
   \)-equidecomposable.
 \end{remark}

On the other hand, a weakened version of the Main Question does hold for arbitrary rational polytopes.
If we allow transformations in \( \glz[d] \ltimes \Q^d \), that is, if we
allow rational translations, then any two Ehrhart-equivalent 
polytopes are equidecomposable~\cite[Prop.~4.3]{HaaseMcAllister}.
Observe, however, that for this group of motions the converse
implication fails: the polytopes $Q$ and $Q'$ from
Example~\ref{example:intervals} are equivalent under \( \glz[1]
\ltimes \Q^1 \), but they have different Ehrhart quasipolynomials.
This sublety goes away if we insist on integral vertices~\cite[Cor.~4.4]{HaaseMcAllister}.

\subsection{Result and structure of proof}

The main result of the present paper is that
Ehrhart-equivalence and \( \glz \)-equidecomposability are the same
for \( 3 \)-dimensional lattice polytopes.

\begin{theorem} \label{thm:main}
  Ehrhart-equivalent lattice \( 3 \)-polytopes are
  \( \glz \)-equidecomposable into half-unimodular simplices.
\end{theorem}

Here a half-unimodular simplex is a simplex whose second dilation is a unimodular simplex (Definition~\ref{defi:half_unimodular}). 

The two ingredients in the proof of Theorem~\ref{thm:main} are a classification of half-unimodular simplices in dimension three (Section~\ref{sec:half_unimodular}) together with the fact that all empty tetrahedra (hence all lattice simplices in $\R^3$) admit a decomposition into relatively open half-unimodular simplices. 
The latter is well-known~\cite{KantorSarkaria,SantosZiegler} but in Section~\ref{sec:empty} we show that the decomposition uses using only half-unimodular simplices of certain types. In Section~\ref{sec:proof} we show that these types have Ehrhart quasipolynomials that are linearly independent in the vector space of all quasipolynomials, which implies that the decompositions constructed in Section~\ref{sec:empty} for Ehrhart-equivalent polytopes $P_1$ and $P_2$ use exactly the same number of half-unimodular simplices of each type, hence providing a $\glz[]$-equidecomposition.

\section{Classification of half-unimodular simplices in $\R^3$}
\label{sec:half_unimodular}
In this section we will give a full classification of half-unimodular simplices under $\Z^3$-equivalence, with the following definition.

\begin{definition}
\label{defi:half_unimodular}
An  $i$-simplex $\Delta$ in $\R^n$ is called \emph{half-unimodular} if $2\Delta$ is a unimodular lattice simplex. That is,
\[
2\Delta \cong \conv(0,e_1,\dots,e_i),
\]
where $e_1,\dots, e_n\in \Z^n$ are the standard basis in $\R^n$
\end{definition}

For any $n$, we have at least the following $2n-1$ half-unimodular simplices:
\begin{align*}
\Delta^1_i&:= \frac12\conv(0, e_1,\dots, e_i), \quad i=\{0,\dots,n\},
\\
\Delta^0_i&:= \frac12\conv(e_1,\dots, e_{i+1}), \quad i=\{0,\dots,n-1\}.
\end{align*}
Observe that the subindex denotes dimension and the superindex $0$ or $1$ denotes the number of lattice points, so these simplices are indeed not Ehrhart-equivalent to one another. 

For $n=3$ there are additionally the following triangle and tetrahedron:
\begin{align*}
\Delta'_2&:= \left(\frac12,\frac12, 0 \right) + \frac12\conv(0, e_1,e_2),\\
\Delta'_3&:= \left(\frac12,\frac12, 0 \right) + \frac12\conv(0, e_1,e_2,e_3).
\end{align*}
None of $\Delta'_2$ and $\Delta'_3$ contain lattice points (which shows they are not equivalent to 
$\Delta^1_2$ and $\Delta^1_3$), and $\Delta'_2$ is distinguished from $\Delta^0_2$ by the fact that its affine span contains lattice points. (In particular, $\Ehr(\Delta^0_2, 2k+1)=0$ for all $k$ while 
$\Ehr(\Delta'_2, 2k+1)>0$ for sufficiently big $k$).

\begin{lemma}
\label{lemma:classification}
Every half-unimodular simplex in $\R^3$ is equivalent to one of the nine defined above.
\end{lemma}

\begin{proof}
We have already justified that the nine simplices above are non-equivalent. Consider now an arbitrary half-unimodular simplex $\Delta$ of dimension $i$ and let's see that it is equivalent to one of the nine.

Note that every half-unimodular simplex  contains at most one integer vertex. 
If $\Delta$ contains an integer point, without loss of generality assume it to be the origin. Then, the unimodular transformation sending $2\Delta$ to $2 \Delta^1_i$ can be chosen to fix the origin, which implies $\Delta$ is equivalent to $\Delta^1_i$.

For the rest let $H$ be the affine span of $\Delta$. If $H$ does not contain integer points (which implies $i\le 2$), let $a\in \Z^3$ be such that $\conv (\{a\} \cup H)\cap \Z^3 = \{a\}$. Then, $\Delta$ can be characterized as being the facet opposite to the unique integer point in $\Delta_a= \conv (\{a\} \cup \Delta)$. Since $\Delta_a$ is, by the previous paragraph, equivalent to $\Delta^1_{i+1}$, this gives an equivalence between $\Delta$ and $\Delta^0_i$.

Then the only case left is when $H$ contains lattice points but $\Delta$ does not, which can only happen if $i\ge 2$.
Observe that, by definition of half-unimodular simplex, $\Delta$ is equivalent to $p+ \Delta_i^1$, for some $p \in \{0,\frac12\}^3$.
If $i=2$ there is no loss of generality in assuming that $H=\R^2\times \{0\}$, so that the only choice of $p$ that makes $\Delta$ not have integer points is indeed $p=(\frac12, \frac12, 0)$.

If $i=3$ we have four possibilities for $p$, namely
\[
p \in \left\{
\left(\frac12,\frac12,0\right),
\left(\frac12,0,\frac12\right),
\left(0,\frac12,\frac12\right),
\left(\frac12,\frac12,\frac12\right)
\right\}. 
\]
It is left to the reader to check that these four possibilities give equivalent simplices.
\end{proof}

\section{Decomposing empty tetrahedra into half-unimodular simplices}
\label{sec:empty}

The main result in this section is that every lattice $3$-polytope
admits a partition into (relatively open) half-unimodular
simplices using only seven of the nine possible types
described in Section~\ref{sec:half_unimodular}.

Since every lattice
polytope can be triangulated into empty simplices, to prove such a
statement we can restrict ourselves to empty tetrahedra:

\begin{definition}
\label{def:empty}
An \emph{empty simplex} is a lattice simplex with no other lattice points
apart of its vertices.
\end{definition}

The classification of empty tetrahedra is classical and relatively simple:

\begin{theorem}[White 1964~\cite{White}]
\label{thm:white}
Every empty tetrahedron of determinant $q$ is unimodularly equivalent to 
\[
\tetra(p,q):= \conv \{ (0,0,0), (1,0,0), (0,0,1), (p,q,1) \},
\]
for some $p\in\Z$ with $\gcd(p,q)=1$. Moreover, $T(p,q)$ is
$\Z$-equivalent to $T(p',q)$ if and only if $p'=\pm p^{\pm 1} \pmod
q$.
\end{theorem}

The most important feature of this classification is the fact that all
empty tetrahedra have \emph{width one}: they are the convex hull of
two edges lying in consecutive parallel lattice planes. In the
 coordinatization of Theorem~\ref{thm:white} (and in the rest of this section) those
planes are $\R^2\times\{0\}$ and $\R^2\times\{1\}$. In particular, the
half-integer points in $\tetra(p,q)$ are:
\begin{enumerate}
\item Its four vertices,
\item The six mid-points of edges, and
\item The following $q-1$ additional points in the parallelepiped
  $Q(p,q)$ with vertices $\frac12(0,0,1), \frac12(p,q,1), \frac12(1,
  0,1)$, and $\frac12({1+p},q,1)$:
\[
a_i:= \frac12\left(\left\lceil \frac{ip}q\right\rceil,i,1\right) =
\frac12\left( \frac{ip+q - (ip\pmod q)}q,i,1\right) 
, \quad i=1,\dots,q-1.
\]
See a picture of $Q(7,12)$ in~Figure~\ref{fig:Q_12_7}
\end{enumerate}

\begin{figure}
\includegraphics[scale=0.3]{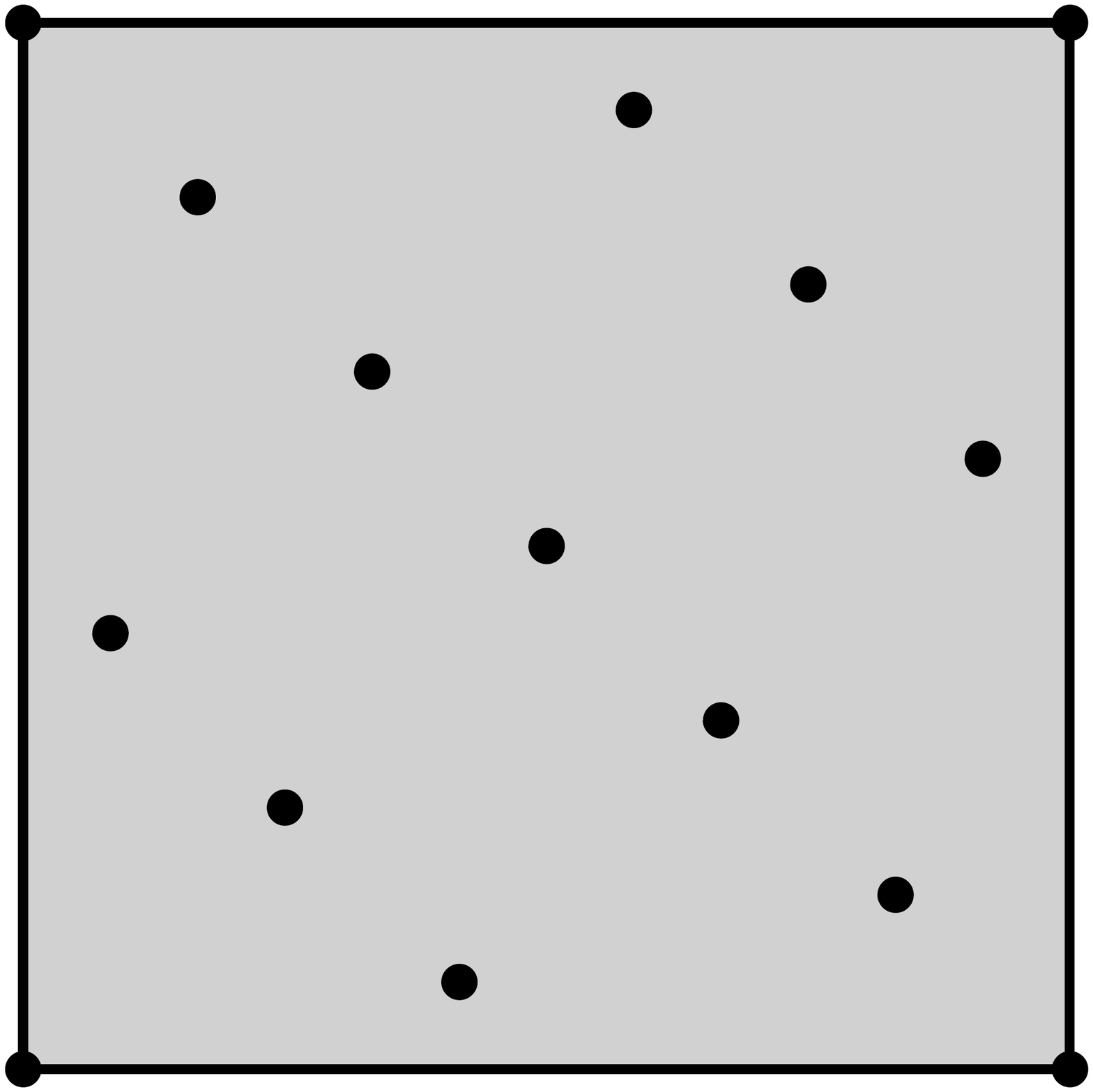}
\caption{A picture of the fundamental parallelogram $Q(7,12)$}
\label{fig:Q_12_7}
\end{figure}

Observe that the second coordinate makes these $a_i$'s form
a naturally ordered sequence. We extend this sequence by setting
$a_0=\frac12(0,0,1)$ and
$a_q=\frac12(p+1,q,1)$ (the choice $a_0=\frac12(1,0,1)$ and
$a_q=\frac12(p,q,1)$ would be equally valid for what follows, but we
do need a choice).

Apart of ordering these $q-1$ points ``vertically''
according to the second coordinate, we can equally order them
``horizontally'' according to the functional $f(x_1,x_2,x_3) =
qx_1-px_2$ that takes the value $0$ on the edge
$\conv\left\{\left(\frac12(0,0,1), \frac12(p,q,1)\right)\right\}$ of $Q(p,q)$ and value $q$
on the opposite edge $\conv\left\{\left(\frac12(1, 0,1),\frac12({1+p},q,1)\right)\right\}$. This
gives a new sequence $\{b_j\}_{j=0}^q$ which coincides as a set with
$\{a_i\}_{i=0}^q$ but where now $f(b_j)=j$. More explicitly, calling
$p'\in\{1,\dots,q-1\}$ the inverse of $-p$ modulo $q$ we have that
\[
b_j:= \frac12\left(\frac{j+(jp'\pmod q)p}q, jp'\pmod q,1\right), \quad
j=1,\dots,q-1.
\]
As before, we extend the sequence with $b_0=\frac12(0,0,1)$ and
$b_q=\frac12(p+1,q,1)$.
See Figure~\ref{fig:Q_12_7}
 for an illustration for $q=12$ and $p=7$, where an affine
transformation has been made so that $Q$ appears as a square in the
picture.

%The main result that we need is:
%
%\begin{lemma}
%\label{lemma:monotone_path}
%For every $i,j\in \{1,q\}$ the following are half-unimodular
%simplices in $\tetra(p,q)$:
%\[
%\conv\left((0,0,0),\left(\frac12, 0,0\right), a_{i-1},a_i\right),
%\quad
%\conv\left(\left(\frac12,0,0\right), (1, 0,0), a_{i-1},a_i\right),
%\]
%and
%\[
%\conv\left((0,0,1),\left(\frac{p}2,\frac{q}2,1\right), b_{j-1}, b_j\right),
%\quad
%\conv\left(\left(\frac{p}2,\frac{q}2,1\right),(p,q,1), b_{j-1}, b_j\right).
%\]
%\end{lemma}

With this we can now prove the main result in this section:

\begin{theorem}
\label{thm:triangulations}
Let $\tetra$ be an empty $3$-simplex, and without loss of generality
assume $\tetra\subset \R^2 \times [0,1]$.
Let $\tetra^-=\tetra \cap  (\R^2 \times [0,\frac12])$ and
$\tetra^+=\tetra \cap  (\R^2 \times [\frac12,1])$ be the two halves of
it. Then, both $\tetra^-$ and $\tetra^+$ have half-unimodular
triangulations in which all tetrahedra contain an integer vertex.
\end{theorem}

\begin{proof}
We triangulate both parts separately. In both cases, observe that we
can consider the quadrilateral $Q=\tetra^-\cap\tetra^+$ as a lattice
parallelogram with $q-1$ interior lattice points and only its vertices
as boundary lattice points, with respect to the lattice $\Lambda=
\frac12\Z^3$.

To triangulate $\tetra^-$, consider the path of vertices $a_0a_1\dots
a_q$, which is monotone with respect to the coordinate $x_2$ and
divides $Q$ into two (non-convex) parts $Q_l$ and $Q_r$. Triangulate
$Q_l$ and $Q_r$ arbitrarily but using all lattice points as vertices,
which gives exactly $2q$ triangles in total (and, although this is
less important, $q$ in each of $Q_l$ and $Q_h$). These two
triangulations, call them ${\mathcal T}_l$ and ${\mathcal T}_r$, are unimodular with respect
to $\Lambda$.
Figure~\ref{fig:tri-Q_12_7} shows them for $q=12$ and $p=7$.

\begin{figure}
\includegraphics[scale=0.15]{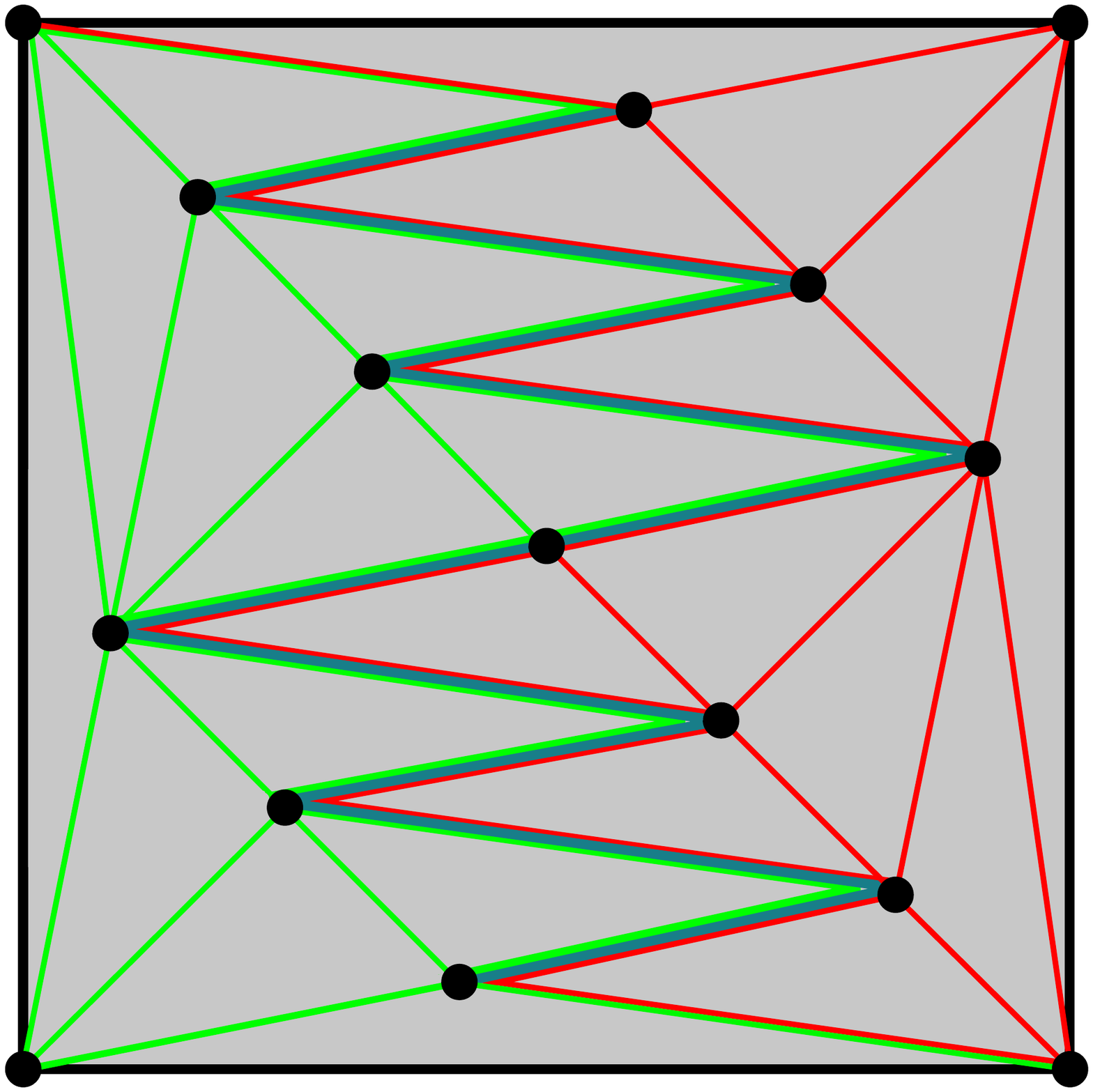}
\includegraphics[scale=0.25]{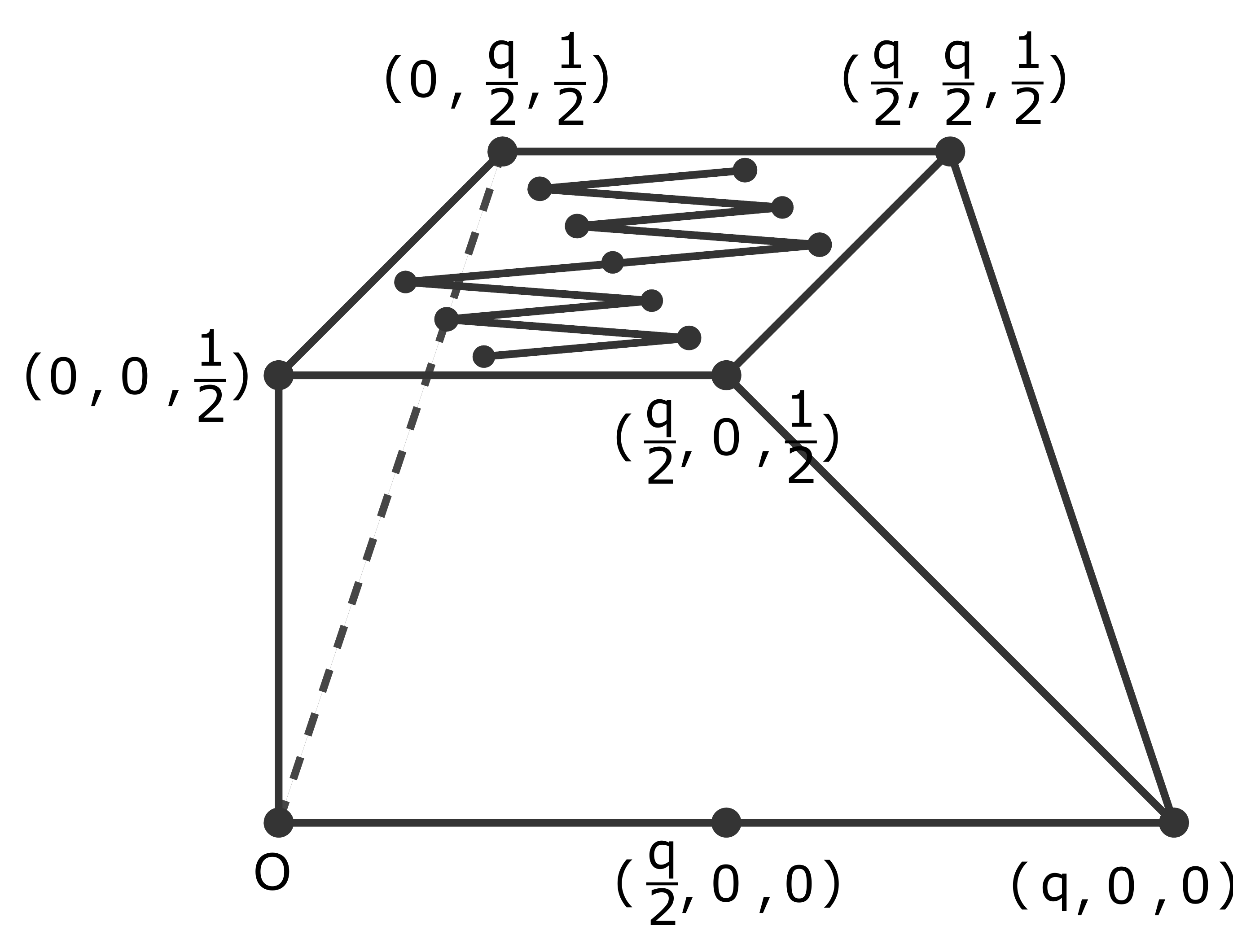}
\caption{A triangulation of the fundamental parallelogram $Q(7,12)$ using the monotone path $a_0a_1\dots a_q$. It is shown both in the quadrilateral $Q(7,12)$ and in the bottom half $\tetra^-$ of $\tetra(7,12)$}
\label{fig:tri-Q_12_7}
\end{figure}

The triangulation $\mathcal T^-$ of $T^-$ consists of:
\begin{itemize}
\item The $q$ tetrahedra obtained joining ${\mathcal T}_l$ to $(0,0,0)$.
\item The $q$ tetrahedra obtained joining ${\mathcal T}_r$ to $(1,0,0)$.
\item The following $2q$ tetrahedra, two for each $ i=1,\dots q$:
% from Lemma~\ref{lemma:monotone_path}:
\[
\conv\left((0,0,0),\left(\frac12, 0,0\right), a_{i-1},a_i\right),
\quad
\conv\left(\left(\frac12,0,0\right), (1, 0,0), a_{i-1},a_i\right)
.
\]
\end{itemize}
All these simplices are half-unimodular: for the first two groups it follows from the fact that they are the join of a half-unimodular triangle and a point at distance $\frac12$ from the hyperplane containing it; for the last group it is an easy calculation to verify it.
It is also clear by construction that each of these tetrahedra contains one (and only one) of the two integer points in $\tetra^-$. 
We omit to proof that $\mathcal T^-$ is a triangulation, but see Remark~\ref{rem:hand_waving} below.
%We leave it to the reader \fs{??} to show that this is a triangulation of $\tetra^-$

For $\tetra^+$ we use the same idea, except now we use the path
$b_0b_1\dots b_q$, which is monotone with respect to the functional
$f$ constant on the second pair of edges of $Q$. Again, this path divides $Q$
into two (non-convex) parts $Q_u$ and $Q_d$ that we triangulate as
before, producing $\mathcal T_u$ and $\mathcal T_d$, and we take as triangulation $\mathcal T^+$ of
$\tetra^+$:
\begin{itemize}
\item The $q$ simplices obtained joining $\mathcal T_d$ to $(0,0,1)$.
\item The $q$ simplices obtained joining $\mathcal T_u$ to $(p,q,1)$.
\item The following $2q$ tetrahedra, two for each $ i=1,\dots q$:
\[
\conv\left((0,0,1),\left(\frac{p}2,\frac{q}2,1\right), b_{j-1}, b_j\right),
\quad
\conv\left(\left(\frac{p}2,\frac{q}2,1\right),(p,q,1), b_{j-1}, b_j\right).
\]
\end{itemize}
\end{proof}

\begin{remark}
\label{rem:hand_waving}
In this proof we skipped some details, in particular the proof that the
sets of tetrahedra $\mathcal T^+$ and $\mathcal T^-$ so defined indeed triangulate $\tetra^-$ and
$\tetra^+$. But these triangulations we construct are nothing but
(scaled down versions of) the lattice triangulations of the upper and
lower halves of $2\tetra(p,q)$ that appear
in~\cite[Sect.~4]{SantosZiegler} and are implicit
in~\cite[Sect.~2]{KantorSarkaria} (see also
\cite[Sect.~9.3.2]{deLoeraRambauSantos} or~\cite[Sect.~4.1]{ut}).
\end{remark}

\begin{corollary}
\label{coro:triangulations}
Every lattice $3$-polytope admits a decomposition into relatively open half-unimodular simplices taken from the following seven classes:
\[
\Delta_0^1, \ \Delta_1^1, \ \Delta_2^1, \ \Delta_3^1, \ 
\Delta_0^1, \ \Delta_1^0, \ \Delta_2^0.
\]
\end{corollary}

\begin{proof}
Let $P$ be a lattice polytope.
First decompose $P$ into relatively open empty simplices. Those of dimensions $0$, $1$, and $2$ can then trivially be decomposed into half-unimodular tetrahedra of the types
\[
\Delta_0^1, \ \Delta_1^1, \ \Delta_2^1.
\]
For the ones of dimension three, use the decomposition coming from the triangulations $\mathcal T^+$ and $\mathcal T^-$ from Theorem~\ref{coro:triangulations}. (To make this a decomposition, consider only the simplices lying in the interior of $\tetra^+$ and $\tetra^-$, plus those from the triangulation of, say,  $\tetra^+$ and lying in the relative interior of $Q=\tetra^+\cap\tetra^-$).

The fact that all tetrahedra in the triangulations $\mathcal T^+$ and $\mathcal T^-$ have a lattice point implies they are all equivalent to $\Delta_3^1$, whose boundary consists of the following types and numbers of relatively open simplices:
\[
3 \Delta_2^1 + 
\Delta_2^0 + 
3 \Delta_1^1 + 
3 \Delta_1^0 + 
\Delta_0^1 + 
3 \Delta_0^0.
\]
\end{proof}

\section{Putting the pieces together}
\label{sec:proof}

For each of the half-unimodular simplices $\Delta_i^j$, $i\in \Z_{\ge 1}$, $j=0,1$ from Section~\ref{sec:half_unimodular} let
\centerline{$E_i^j(k) \coloneqq  \# (k\relint(\Delta_i^j) \cap \Z^3)$}
denote the Ehrhart function of the \emph{relative interior} of
$\Delta_i^j$, which is a quasipolynomial of period two.

\begin{proposition}
\label{prop:basis}
The $2n+1$ quasipolynomials 
\[
\{E_i^1: i=0,\dots,n\} \cup \{E_i^0: i=0,\dots,n-1\} 
\]
form a basis for the linear span 
of all Ehrhart quasipolynomials of half-lattice polytopes in $\R^n$.
\end{proposition}

\begin{proof}
Ehrhart quasipolynomials of half-lattice polytopes of dimension at most $n$ have period  two and degree at most $n$, so they can  be written as linear combinations of the following $2n$ quasi-monomials:
\[
1,k,k \dots, k^n, 
(-1)^k,(-1)^kk,\dots, (-1)^kk^n.
\]
But the quasi-monomial  $(-1)^kk^n$ is not used by polytopes in $\R^n$, because the coefficient of degree $n$ in $\Ehr(P;k)$ for a $P \subset \R^n$ equals the $n$-dimensional volume of $P$. This implies that Ehrhart quasipolynomials of half-lattice polytopes in $\R^n$ generate a vector space of dimension at most $2n-1$. The $E_i^0$'s and $E_i^1$'s are independent, since there are two of each degree $i$ and their leading terms are different (see, e.g., the remark below). Hence they form a basis for this vector space.
\end{proof}

\begin{proof}[Proof of Theorem~\ref{thm:main}]
Given a lattice $3$-polytope $P$, we can decompose it into relatively
open half-unimodular simlices of the form $\Delta_i^0$ or $\Delta_i^1$
by Corollary~\ref{coro:triangulations}.
Denoting by $f_i^j$ the number of simplices of type $\Delta_i^j$
in this decomposition for each $i\in\{0,1,2,3\}$ and $j\in\{0,1\}$,
we can write the Ehrhart quasipolynomial of $P$ as
\[
\Ehr(P;k) = \sum_{i,j} f_i^j E_i^j(k)\,.
\]

Conversely, by Proposition~\ref{prop:basis}, if $P$ and $P'$ have the same Ehrhart quasipolynomials then the coefficients in the above expression for $P$ and $P'$ are the same, so that the decompositions of Corollary~\ref{coro:triangulations} for $P$ and $P'$ are in fact $\glz[]$-equi\-decompo\-si\-tions of them.
\end{proof}

\begin{remark}
The Ehrhart quasipolynomials $E_i^0$ and $E_i^1$ admit simple closed formulas.
For the even values of $k$, being a half-unimodular simplex implies that $E_i^j(2k) = \binom{k-1}{i}$. For the odd values of $k$:  
\begin{itemize}
\item For $E_i^0$, the fact that the affine span of $\Delta_i^0$ contains no lattice points implies 
$E^0(2k-1) = 0$. 
Together with the previous fact, this gives:
  \begin{equation}
  \label{eq:quasi0}
  E_i^0(k)= \frac{1+(-1)^k}{2}  \binom{\lceil k/2\rceil-1}{i}.
  \end{equation}

\item For $E_i^1$ we have $E_i(2k) = E_i^1(2k-1)$, which implies
  \begin{equation}
  \label{eq:quasi1}
  E_i^1(k)=  \binom{\lceil k/2\rceil-1}{i}.
  \end{equation}
The reason is that the facet of $(2k-1)\Delta_i^1$ opposite to its unique lattice vertex does not contain lattice points, so enlarging it to $2k\Delta_i^1$ does not add interior lattice points. (See Figure~\ref{fig:delta_31} for an illustration of $7\Delta_3^1$).
\end{itemize}

\begin{figure}
\includegraphics[scale=0.3]{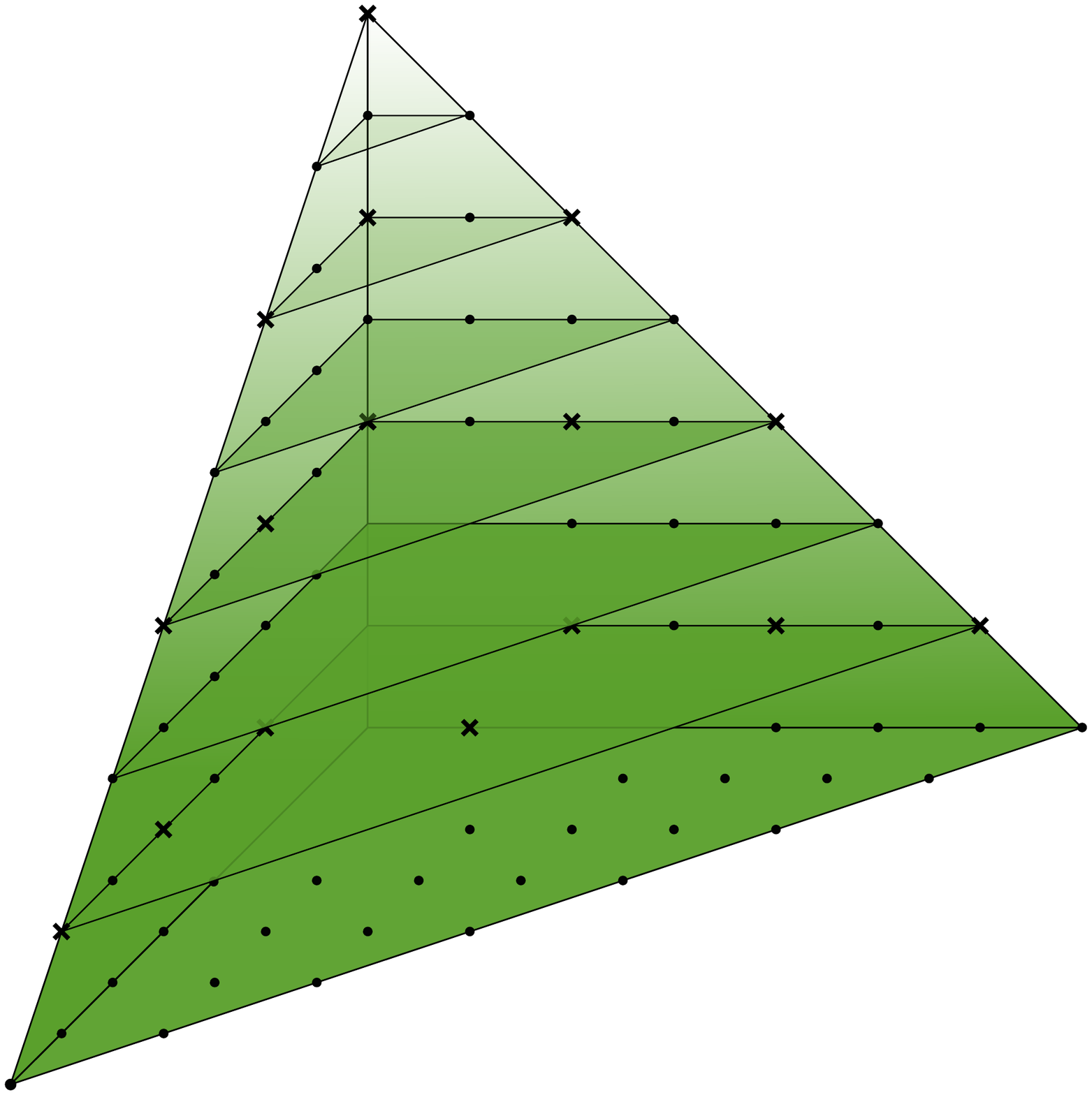}
\caption{A picture of $7\Delta_3^1$. Lattice points appear as crosses and half-lattice points as dots}
\label{fig:delta_31}
\end{figure}

These closed forms readily show that they are independent, since their evaluations start with:
\begin{align*} \label{eq:lattice-basis}
E_0^1 (k) &= \ \ 1\ \ \dots \\
E_0^0 (k) &= \ \ 0\ \ 1\ \ \dots\\
E_1^1 (k) &= \ \ 0\ \ 0\ \ 1\ \ \dots\\
E_1^0 (k) &= \ \ 0\ \ 0\ \ 0\ \ 1\ \ \dots\\
E_2^1 (k) &= \ \ 0\ \ 0\ \ 0\ \ 0\ \ 1\ \ \dots\\
E^0_2 (k) &= \ \ 0\ \ 0\ \ 0\ \ 0\ \ 0\ \ 1\ \ \dots\\
E_3^1 (k) &= \ \ 0\ \ 0\ \ 0\ \ 0\ \ 0\ \ 0\ \ 1\ \ \dots\ .
\end{align*} 
Observe that all the above values can also be checked in appropriate faces of $k\Delta_3^1$, $k=1,\dots,7$ (compare Figure~\ref{fig:delta_31}).
Incidentally, these evaluations (which generalize naturally to higher dimensions) imply that the quasipolynomials of Proposition~\ref{prop:basis} form an integer basis (and not only a linear basis as stated in the proposition)
for the lattice of integer-valued period-two quasipolynomials, which includes all Ehrhart quasipolynomials of half-lattice polytopes.

\end{remark}

\begin{remark}
Formulas~\eqref{eq:quasi0} and~\eqref{eq:quasi1},
together with direct calculations for $E_2'(2k+1)$ and $E_3'(2k+1)$, give us the following  formulas for the nine Ehrhart quasipolynomials of half-unimodular simplicies in $\R^3$ in terms of the quasi-monomials mentioned in the proof of Proposition~\ref{prop:basis}:

\begin{itemize}
\item[] $E_0^1(k)=1$
\item[] $E_1^1(k)=\frac{1}{4}(2k+(-1)^{k}-1)$
\item[] $E_2^1(k)=\frac{1}{16}\left(2k^2-2\left((-1)^{k}+5\right)k+5(-1)^k+11\right)$
\item[] $E_3^1(k)=\frac{1}{96}\left(2k^3-3\left((-1)^{k}+7\right)k^2+\left(21(-1)^k+67\right)k-33(-1)^k-63\right)$
\medskip

\item[] $E_0^0(k)=\frac{1 + (-1)^{k}}2$
\item[] $E_1^0(k)=\frac{1+(-1)^{k}}{4}  (k-2)$
\item[] $E_2^0(k)=\frac{1+(-1)^{k}}{16}(k^2-6k+8)$
\medskip

\item[] $E_2'(k)=\frac{1}{16}\left(2k^2-6\left((-1)^{k}+1\right)k+9(-1)^k+7\right)$
\item[] $E_3'(k)=\frac{1}{96}\left(2k^3-\left(15(-1)^{k}+9\right)k^2+\left(45(-1)^k+43\right)k-51(-1)^k-45\right)
  $
\end{itemize}
\end{remark}

\bibliographystyle{plain}
\bibliography{3d-equidecomposable}

\begin{thebibliography}{10}

\bibitem{BaranyKantor}
Imre B\'ar\'any and Jean-Michel Kantor.
\newblock Universal counting of lattice points in polytopes.
\newblock {\em Publ. Inst. Math. (Beograd) (N.S.)}, 66(80):16--22, 1999.
\newblock Geometric combinatorics (Kotor, 1998).

\bibitem{BetkeMcmullen}
Ulrich Betke and Peter McMullen.
\newblock Lattice points in lattice polytopes.
\newblock {\em Monatsh. Math.}, 99(4):253--265, 1985.

\bibitem{deLoeraRambauSantos}
Jes\'us~A. De~Loera, J\"org Rambau, and Francisco Santos.
\newblock {\em Triangulations: Structures for Algorithms and Applications}.
\newblock Springer, 2010.

\bibitem{EhrhartBook}
Eug{\`e}ne Ehrhart.
\newblock {\em Polyn{\^o}mes arithm{\'e}tiques et m{\'e}thode des poly{\`e}dres
  en combinatoire}, volume~35 of {\em Int.\ ser.\ numerical math.}
\newblock Birkh{\"a}user, 1977.

\bibitem{Greenberg}
Peter Greenberg.
\newblock Piecewise {SL$_2(\mathbb{Z})$}-geometry.
\newblock {\em Trans. AMS}, 335(2):705--720, 1993.

\bibitem{HaaseMcAllister}
Christian Haase and Tyrrell McAllister.
\newblock Quasi-period collapse and {GL$_n(\mathbb{Z})$}-scissors congruence in
  rational polytopes.
\newblock In Matthias Beck and et. al., editors, {\em Integer points in
  polyhedra}, volume 452 of {\em Contemp. Math.}, pages 115--122. AMS, 2008.
\newblock Geometry, number theory, representation theory, algebra,
  optimization, statistics. Proceedings of an AMS-IMS-SIAM conference,
  Snowbird, UT, USA, June 11--15, 2006.

\bibitem{ut}
Christian Haase, Andreas Paffenholz, Lindsay~C.\ Piechnik, and Francisco
  Santos.
\newblock Existence of unimodular triangulations -- positive results.
\newblock \url{arXiv:1405.1687}, 2014.

\bibitem{Kantor}
Jean-Michel Kantor.
\newblock Triangulations of integral polytopes and {E}hrhart polynomials.
\newblock {\em Beitr\"age Algebra Geom.}, 39(1):205--218, 1998.

\bibitem{KantorSarkaria}
Jean-Michel Kantor and Karanbir~S.\ Sarkaria.
\newblock On primitive subdivisions of an elementary tetrahedron.
\newblock {\em Pacific J. Math.}, 211(1):123--155, 2003.
\newblock IHES Preprint
  \url{http://www.ihes.fr/PREPRINTS/M01/Resu/resu-M01-23.html}.

\bibitem{TurnerWuAlgorithm}
Yuhuai~Wu Paxton~Turner.
\newblock Conditions for discrete equidecomposability of polygons.
\newblock \href{arxiv.org/abs/1412.0191}{arXiv:1412.0191 [math.CO]}, Nov 2014.

\bibitem{TurnerWuWeightInvariant}
Yuhuai~Wu Paxton~Turner.
\newblock Discrete equidecomposability and ehrhart theory of polygons.
\newblock \href{arxiv.org/abs/1412.0196}{arXiv:1412.0196 [math.CO]}, Nov 2014.

\bibitem{SantosZiegler}
Francisco Santos and G\"unter~M. Ziegler.
\newblock Unimodular triangulations of dilated $3$-polytopes.
\newblock {\em Trans. Moscow Math. Soc.}, 74:293--311, 2013.

\bibitem{White}
G.~K. White.
\newblock Lattice tetrahedra.
\newblock {\em Canad. J. Math.}, 16:389--396, 1964.

\end{thebibliography}

\end{document}